\documentclass[a4paper, 11pt]{article}
\usepackage[mathscr]{eucal}
\usepackage{amssymb}
\usepackage{latexsym}
\usepackage{amsthm}
\usepackage{amsmath}
\usepackage[pdftex]{graphicx}
\usepackage{color}
\usepackage{mathrsfs}
\usepackage{psfrag}
\usepackage{a4wide}
\usepackage{comment}

\DeclareMathOperator{\rank}{rank}
\newtheorem{theorem}{Theorem}[section]
\newtheorem{proposition}[theorem]{Proposition}
\newtheorem{lemma}[theorem]{Lemma}
\newtheorem{corollary}[theorem]{Corollary}
\theoremstyle{definition}

\makeatletter
\@addtoreset{equation}{section}

\makeatother

\title{A proof of a Dodecahedron conjecture for distance sets}  

\author{
Hiroshi Nozaki and Masashi Shinohara
}
\begin{document}
\maketitle

\renewcommand{\thefootnote}{\fnsymbol{footnote}}
\footnote[0]{2010 Mathematics Subject Classification: 
 05D05 (05B05)  }

\begin{abstract}
A finite subset of a Euclidean space is called an $s$-distance set if 
there exist exactly $s$ values of the Euclidean distances between two distinct points in the set. 
In this paper, we prove that 
the maximum cardinality among all 5-distance sets in $\mathbb{R}^3$ is 20, and every $5$-distance set in $\mathbb{R}^3$ with $20$ points 
is similar to the vertex set of a regular dodecahedron. 
\end{abstract}

\textbf{Key words}: 
Distance sets, dodecahedron. \\
\section{Introduction}

For $X\subset \mathbb{R}^d$, let 
\[A(X)=\{d(x,y)\mid x,y\in X, x\ne y\},\]
where $d(x,y)$ is the Euclidean distance between $x$ and $y$. 
We call $X$ an {\it $s$-distance set} if $|A(X)|=s$. 
Two $s$-distance sets are said to be {\it isomorphic} if there exists a similar 
transformation from one to the other. 
One of the major problems in the theory of distance sets is 
to determine the maximum cardinality $g_d(s)$ of $s$-distance sets in $\mathbb{R}^d$ 
for given $s$ and $d$, and classify distance sets in $\mathbb{R}^d$ with $g_d(s)$ points up to isomorphism.  
An $s$-distance set $X$ in $\mathbb{R}^d$ is said to be {\it optimal} if $|X|=g_d(s)$. 
Clearly $g_1(s)=s+1$, and the optimal $s$-distance set is the set of $s+1$ points on the line whose two consecutive points have an equal interval. 
For the cases where $d=2$ or $s=2$, $s$-distance sets in $\mathbb{R}^d$ are well studied \cite{BBS, B84,ES66,EF96,LRS77,L97,S04,S08,W12}, 
because of their simple structures or the relationship to graphs, see Table~\ref{tb:1}. 
For $d\leq 8$, the maximum cardinality $g_d(2)$ are determined, and optimal $2$-distance sets in $\mathbb{R}^d$ are classified except for $d=8$ \cite{ES66,L97}. 
Moreover, it is known that $g_3(3)=12$, $g_3(4)=13$ and $g_4(4)=16$, and the classification is complete for the three cases \cite{Spre,SO20}. 
In particular, we recall the classification of optimal $s$-distance sets in $\mathbb{R}^d$ for $(d,s)=(2,4), (3,3)$ and $(3,4)$ 
as in Theorem~\ref{thm:known}. 

\begin{center}
\begin{table}[h] \label{tb:1}
\begin{center}
\begin{tabular}{cccccccc}
\hline 
$d$ & $2$ &  $3$  & $4$ & $5$ &  $6$  & $7$ & $8$   \\
\hline 
$g_d(2)$ & $5$ & $6$ & $10$ & $16$ & $27$ & $29$ & $45$\\
\hline 
\end{tabular}\qquad 
\begin{tabular}{cccccccc}
\hline 
$s$ & $2$ &  $3$  & $4$ & $5$ &  $6$   \\
\hline 
$g_2(s)$ & $5$ & $7$ & $9$ & $12$ & $13$\\
\hline 
\end{tabular}
\caption{Maximum cardinalities of $s$-distance sets in $\mathbb{R}^d$}
\end{center}
\end{table} 
\end{center}

\begin{theorem}\label{thm:known}{{\rm(\cite{S08,Spre, SO20})}}
\begin{itemize}
\item[{\rm (1)}] Every $9$-point $4$-distance set in $\mathbb{R}^2$ is isomorphic to the  vertices of the regular 
nonagon or one of the three configurations given in Figure~\ref{fig:img}\ (a)--(c). 
Moreover, every $8$-point $4$-distance set in $\mathbb{R}^2$ is isomorphic to the vertices of the regular 
octagon, the vertices of the regular septagon with its center, Figure~\ref{fig:img}\ (d) or 
$8$-point subsets of a $9$-point $4$-distance set. 
\item[{\rm (2)}] Every $12$-point $3$-distance set in $\mathbb{R}^3$ is isomorphic to the vertices of 
the icosahedron. 
\item[{\rm (3)}] Every $13$-point $4$-distance set in $\mathbb{R}^3$ is isomorphic to the vertices of 
the icosahedron with its center point or the vertex set of the cuboctahedron with its center point. 
\end{itemize}
\end{theorem}

\begin{figure}[htbp]
\begin{center}
  \begin{tabular}{c}

    \begin{minipage}{0.225\hsize}
      \begin{center}
        \includegraphics[clip, width=32mm]{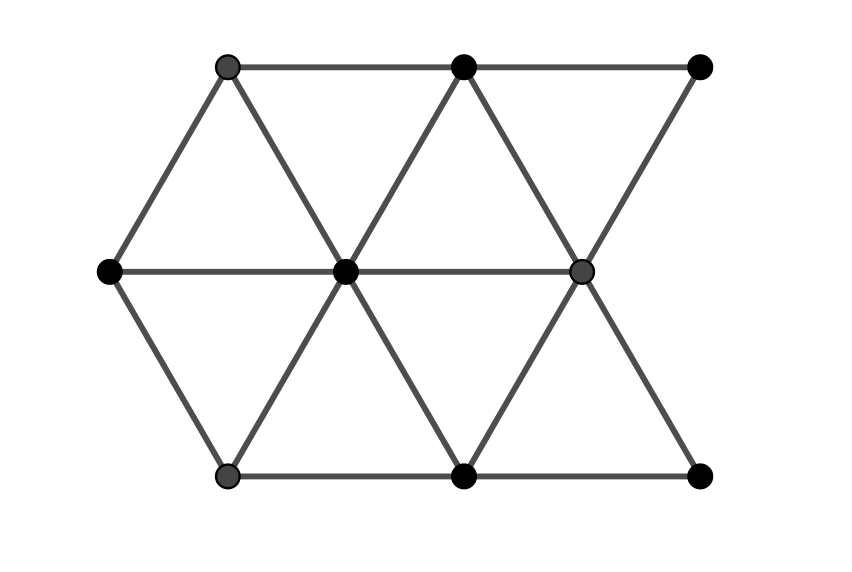}
      \end{center}
    \end{minipage}

    \begin{minipage}{0.225\hsize}
      \begin{center}
        \includegraphics[clip, width=32mm]{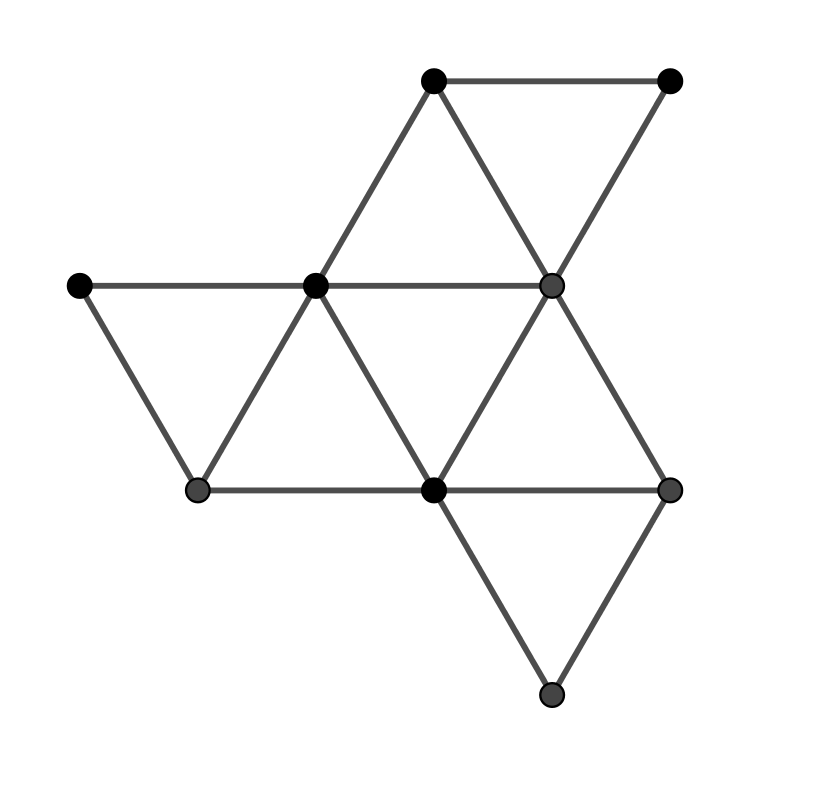}
      \end{center}
    \end{minipage}

    \begin{minipage}{0.225\hsize}
      \begin{center}
        \includegraphics[clip, width=30mm]{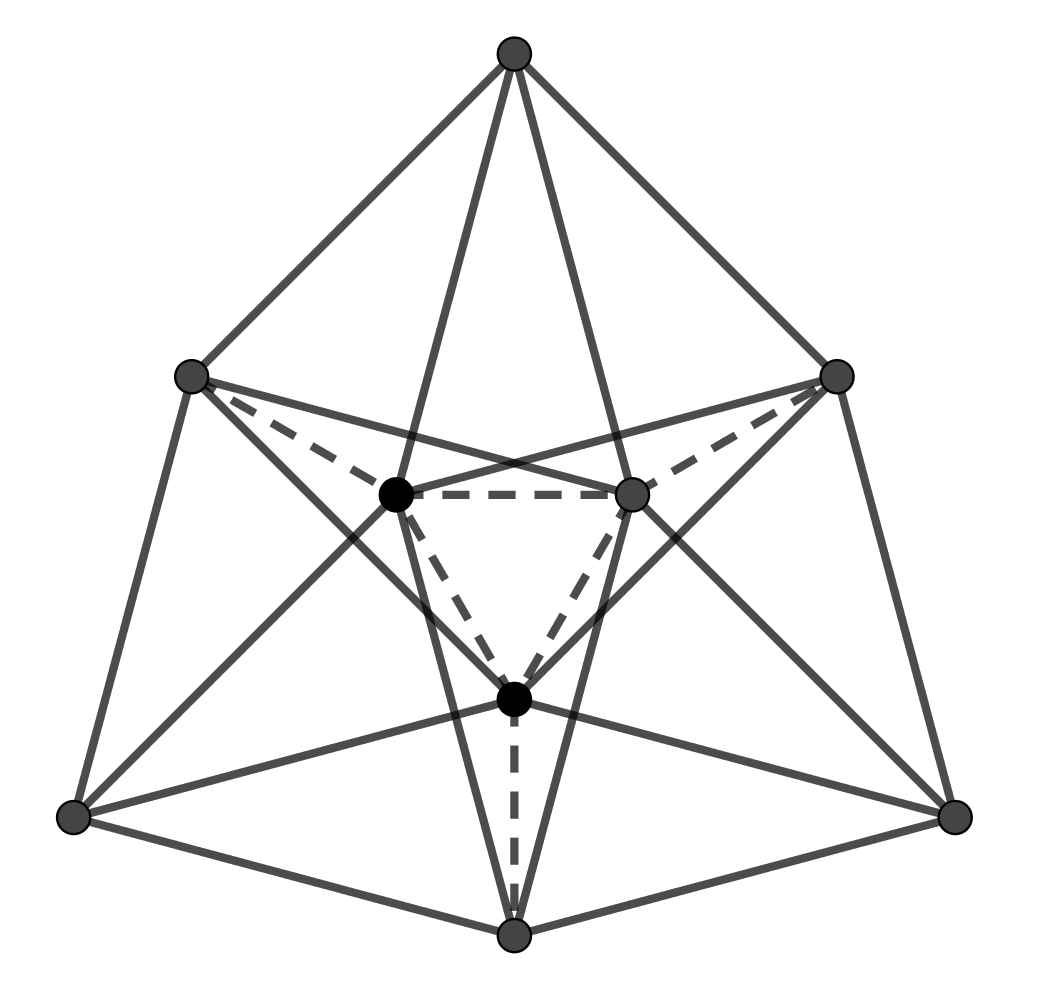}
      \end{center}
    \end{minipage}

    \begin{minipage}{0.225\hsize}
      \begin{center}
        \includegraphics[clip, width=30mm]{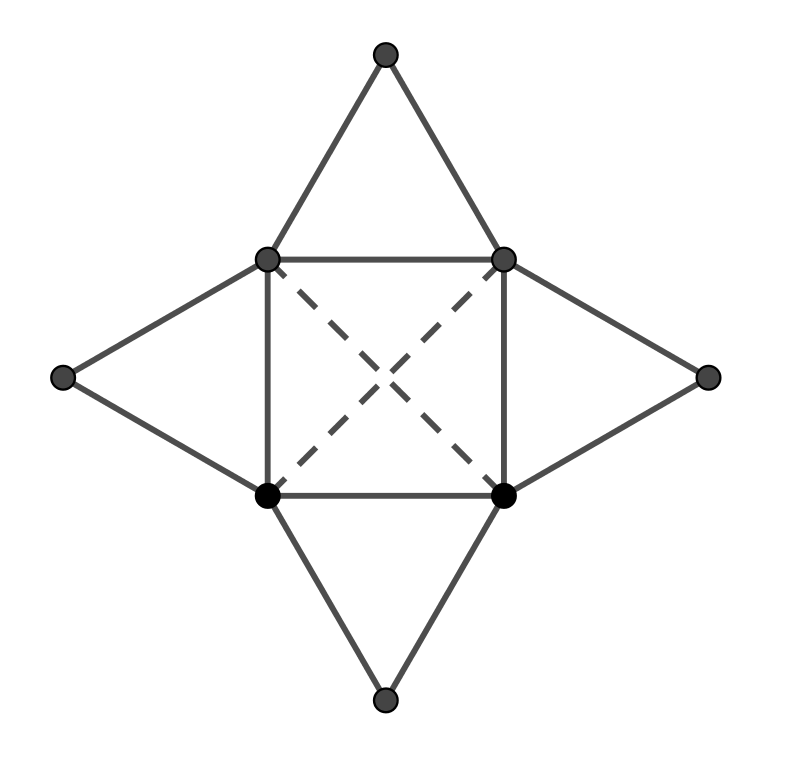}
      \end{center}
    \end{minipage}\\
    
     \hspace{-0.3cm}(a) \hspace{2.9cm} (b) \hspace{2.9cm} (c) \hspace{2.8cm} (d)
  \end{tabular}
  \caption{Maximal planar $4$-distance sets}
  \label{fig:img}
  \end{center}
\end{figure}

For a $2$-distance set $X$, we consider the graph on $X$ where two vertices are adjacent if they have the smallest distance in $X$. 
We can construct the $2$-distance set that has the structure of a given graph \cite{ES66}. 
Lison\v{e}k \cite{L97} gave an algorithm for a stepwise augmentation of representable graphs (adding one vertex per iteration), and classified the optimal 2-distance sets in $\mathbb{R}^d$ for $d\leq 7$ by a computer search.  Sz\"{o}ll\H{o}si and {O}sterg{\aa}rd \cite{SO20} extended this algorithm to $s$-distance sets and classified optimal $s$-distance sets for $(d,s)=(2,6), (3,4), (4,3)$. 
Indeed, their algorithm is applicable for small $s$ and $d$. 
In the present paper, we add geometrical observations in $\mathbb{R}^3$ to this algorithm, and obtain the main theorem as follows.  

\begin{theorem}\label{thm:main}
Every $20$-point $5$-distance set in $\mathbb{R}^3$ is isomorphic to the vertices of a regular 
dodecahedron. In particular, $g_3(5)=20$. 
\end{theorem}
This was a long standing open problem \cite{CFG94} as well as the icosahedron conjecture \cite{ES66}. 
The icosahedron conjecture was already solved, and the set is the optimal 3-distance set in $\mathbb{R}^3$ \cite{Spre,SO20} as in Theorem~\ref{thm:known} (2). 
The following theorem plays a key role to prove our main theorem. 

\begin{theorem}\label{thm:sub}
Every $5$-distance set in $\mathbb{R}^3$ with at least $20$ points contains 
an $s$-distance set for some $s\le 4$ with $8$ points. 
\end{theorem}

The main concept to prove Theorem~\ref{thm:sub} is the diameter graph \cite{Dol} of a subset in $\mathbb{R}^d$. 
The diameter graph of a set $X$ in $\mathbb{R}^d$ is the graph on $X$ where two vertices are adjacent if the two vertices have the largest distance in $X$. 
The subset of $X$ corresponding to an independence set of the diameter graph does not have the largest distance in $X$. Thus we can verify the existence of an $s'$-distance subset of an $s$-distance set $X$ with $s'<s$ by the independence number of its diameter graph. The existence of $s'$-distance set is useful to determine an optimal $s$-distance set in low dimensions \cite{S08,Spre}. 
Ramsey numbers or complementary Ramsey numbers \cite{MS19} are also expected to show the existence of an $s'$-distance subsets of an $s$-distance set.  

In section 2, we discuss the distances in a regular dodecahedron and 
we enumerate the number of $8$-point subsets of a dodecahedron which are $3$- or $4$-distance. 
In section 3, we consider the independence numbers of diameter graphs and prove Theorem~\ref{thm:sub}.  
The classification of $8$-point $s$-distance sets in $\mathbb{R}^3$ for $s\le 4$ are essentially obtained by 
Sz\"{o}ll\H{o}si and \"{O}sterg{\aa}rd \cite{SO20}. 
In section 4, we introduce their methods, where $s$-distance sets are constructed from $s$-colorings. 
In section 5, we classify 8-point $3$- or $4$-distance sets which may be subsets of a 20-point 5-distance set in $\mathbb{R}^3$, and prove Theorem~\ref{thm:main}.

\section{Dodecahedron and its subsets}

Let $G=(V,E)$ be a simple graph, where $V=V(G)$ and $E=E(G)$ are the vertex set and the edge set of $G$, respectively. 
A subset $W$ of $V(G)$ is an {\it independent set} ({\it resp.} {\it clique}) of $G$ 
if any two vertices in $W$ are nonadjacent ({\it resp.} adjacent). 
The {\it independence number} $\alpha (G)$ ({\it resp.} {\it clique number} $\omega(G)$) of a graph $G$ 
is the maximum cardinality among the independent sets ({\it resp.} cliques) of $G$. 
Let $R_i=\{ (x,y) \in V\times V \mid  \mathfrak{d}(x,y)=i \}$, where $\mathfrak{d}$ is the shortest-path distance.    
The {\it $i$-th distance matrix} $A_i$ of $G$ is the matrix indexed by $V$ whose $(x,y)$-entry is 
$1$ if $(x,y) \in R_i$, and 0 otherwise.  
A simple graph $G$ is a {\it distance-regular graph} \cite{BCNb, DKT16} if for any non-negative integers $i,j,k$, 
the number $p_{ij}^k=|\{z \in V \mid  (x,z) \in R_i, (z,y) \in R_j \}|$ is independent of the choice of $(x,y) \in R_k$. 
The algebra $\mathfrak{A}$ spanned by $\{A_i \}$ over the complex numbers is called 
the {\it Bose--Mesner algebra} of a distance-regular graph. 
There exists another basis $\{E_i\}$ such that $E_iE_j= \delta_{ij}E_i$, where $\delta_{ij}$ is the Kronecker delta. The matrices $E_i$ are called {\it primitive idempotents}, and the matrices are positive semidefinite. 
The matrices $E_i$ can be interpreted as the Gram matrices of some spherical sets that have the structure of the distance-regular graph, and $E_i$ are called {\it spherical representations} of the graph. 
The following matrices 
\begin{align*}
P=(p_i(j))_{j,i}&  \text{ for $A_i=\sum_{j} p_i(j) E_j$}, \\
Q=(q_i(j))_{j,i}&  \text{ for $E_i=\frac{1}{|V|} \sum_{j} q_i(j) A_j$}, 
\end{align*}
are called the {\it first} and {\it second eigenmatrices}, respectively.  
The entries of $P$ are the eigenvalues of $A_i$, and the entries 
of $Q$ are the inner products of the spherical representation of $E_i$. 
The first row $q_i(0)$ of $Q$ is the rank of $E_i$, that is the dimension where the representation $E_i$ exists.

Let $\mathcal{D}_{20}$ be the vertex set of the dodecahedron with edge length $1$. The set $\mathcal{D}_{20}$ is a 5-distance set, and let $d_1=1,d_2,d_3,d_4,d_5$ be the 5-distances of $\mathcal{D}_{20}$ with $1=d_1<d_2<d_3<d_4<d_5$. The second-smallest distance $d_2$ is the length of a diagonal line of a face, namely $d_2=\tau=(1+\sqrt{5})/2$. 
Since $\mathcal{D}_{20}$ contains the cube with edge length $\tau$, 
the other distances in the cube are  
$d_3=\sqrt{2}\tau$ and $d_5=\sqrt{3}\tau$. 
We can calculate $d_4=\sqrt{3\tau^2-1}=\tau+1$ by Pythagorean theorem. 
  Let $\mathfrak{G}$ be the dodecahedron graph $\mathfrak{G}=(V, E)$, where $V=\mathcal{D}_{20}$ and $E=\{(x,y) \mid d(x,y)=d_1\}$. 
  The graph $\mathfrak{G}$ is a distance-regular graph, and 
  $d(x,y)=d_i$ if and only if $\mathfrak{d}(x,y)=i$ for each $i\in \{0,1,\ldots,5\}$, where $d_0=0$.  
  The second eigenmatrix $Q$ of $\mathfrak{G}$ is 
 \[
 Q=
 \begin{pmatrix}
 1 & 3 & 3 & 4 & 4 & 5 \\
  1 & \sqrt{5} & -\sqrt{5} & -8/3 & 0 & 5/3 \\
 1 & 1 & 1 & 2/3 & -2 & -5/3 \\
 1 & -1 & -1 & 2/3 & 2 & -5/3 \\
  1 & -\sqrt{5} & \sqrt{5} & -8/3 & 0 & 5/3 \\
 1 & -3 & -3 & 4 & -4 & 5 \\
\end{pmatrix}.
 \]
There are two representations $E_2$ and $E_3$ in the 3-dimensional sphere. 
Indeed, both $E_2$ and $E_3$ are the dodecahedron, and the two graphs of $A_1$ and $A_4$ are isomorphic. 
Let $\Phi$ be the field automorphism of $\mathbb{Q}(\sqrt{5})$ such that $\Phi(\sqrt{5})=-\sqrt{5}$ and $\Phi$ fixes all rationals. 
For a matrix $M=(m_{ij})$ with $m_{ij} \in \mathbb{Q}(\sqrt{5})$, the map $\hat{\Phi}(M)$ is defined by applying $\Phi$ to the entries of $M$, namely $\hat{\Phi}(M)=(\Phi(m_{ij}))$. 
It follows that 
\begin{align*}
\hat{\Phi}(E_2)&=3A_0+\Phi(\sqrt{5})A_1+A_2-A_3-\Phi(\sqrt{5})A_4-3A_5\\
&=3A_0-\sqrt{5}A_1+A_2-A_3+\sqrt{5}A_4-3A_5=E_3.
\end{align*}
A principal submatrix $T$ of $E_2$ corresponds to a subset of the dodecahedron. The matrix $\hat{\Phi}(T)$ is a principal submatrix of $E_3$, and $\hat{\Phi}(T)$ also corresponds to a subset of the dodecahedron. 
The two matrices $T$ and $\hat{\Phi}(T)$ may not be isomorphic as distance sets, 
but the two colorings of them are equivalent (see Section~\ref{sec:5} for colorings).  
This observation gives the following lemma.  
\begin{lemma}\label{lem:structures}
Let $X$ be a subset of the dodecahedron in the unit sphere $S^2$. 
Let $M$ be the Gram matrix of $X$. Let $\hat{\Phi}$ is the map defined as above. Then $M$ and $\hat{\Phi}(M)$ are subsets of the dodecahedron, and the two colorings of them are equivalent. 
\end{lemma}
Now we discuss 8-point subsets of the dodecahedron which have only 3 or 4 distances. 
\begin{lemma}\label{lem:SubDodeca3ds}
There exists unique $3$-distance subset of a regular dodecahedron with $8$ points up to isomorphism. The subset is the cube.  
\end{lemma}
\begin{proof}
Let $\mathfrak{G}$ be the dodecahedron graph with relations $R_i=\{(x,y) \in V \times V \mid d(x,y)=d_i\}$.  
We define the graphs $\mathfrak{G}_i=(V,R_i)$ and $\mathfrak{G}_{i,j}=(V, R_i \cup R_j)$.  
If for given $i$, the independence number $\alpha(\mathfrak{G}_{i,j})$ is less than $8$ for each $j\ne i$, then we should take the distance $d_i$ for a 8-point subset. 
We can determine $\alpha(\mathfrak{G}_2) =6$ and $\alpha(\mathfrak{G}_3)=5$. This implies $\alpha(\mathfrak{G}_{2,i})\leq 6 <8$ and $\alpha(\mathfrak{G}_{3,i})\leq 5<8$ for each $i=1,4,5$. Thus $X$ has both $d_2$ and $d_3$. 
Moreover, we can determine $\alpha(\mathfrak{G}_{1,5})=\alpha(\mathfrak{G}_{4,5})=7$ and $\alpha(\mathfrak{G}_{1,4})=8$, which are calculated by a computer aid. Therefore the distances of $X$ are $d_2$, $d_3$ and $d_5$. The set $X$ corresponding to $\alpha(\mathfrak{G}_{1,4})$ is the cube.  
\end{proof}

\begin{lemma}\label{lem:SubDodeca4ds}
There exist exactly $116$ of $4$-distance subsets of a regular dodecahedron with $8$ points up to isomorphism. 
\end{lemma}
\begin{proof}
An 8-point 4-distance subset of a regular dodecahedron does not contain an antipodal pair $\{x,-x\}$, otherwise $X$ is not 4-distance. 
A regular dodecahedron has only 10 antipodal pairs. 
We choose 8 antipodal pairs from the 10 pairs, and pick out one point from each antipodal pair, then an 8-point 4-distance set is obtained. Every 8-point 4-distance subset of a regular dodecahedron is obtained by this manner. 

First we prove that 
if two 8-point 4-distance sets $X$ and $Y$ are isomorphic, then $X$ and $Y$ are in the same orbit of the isometry group of a regular dodecahedron. 
Since the sets $X$ and $Y$ are in the same sphere, there exists an isometry $\sigma$ in the orthogonal group $O(3)$ such that $X^\sigma=Y$. 
This implies that $(\pm X)^{\sigma}= \pm Y$, namely the set of 8 antipodal pairs of $\pm X$ are isomorphic to that of $\pm Y$. 
Since a regular dodecahedron in a given sphere is uniquely determined after one  face is fixed, 
if each set of 8 antipodal pairs makes a face of the dodecahedron, then $\sigma$ becomes an isometry of the dodecahedron. 
In order to prove that each set of 8 antipodal pairs makes a face of the dodecahedron, we prove that it is impossible to break all the faces of a regular dodecahedron by removing 2 antipodal pairs.  
If we remove an antipodal pair, then 6 faces are broken. 
By removing one more antipodal pair, we would like to break the remaining 6 faces, but it is impossible. Therefore, 
each set of 8 antipodal pairs contains a face of the dodecahedron, and $\sigma$ becomes an isometry of the dodecahedron. 

We can determine the number of the 4-distance sets up to isomorphism by Burnside's lemma. The isometry group ${\rm Aut}(\mathcal{D}_{20})$ of a regular dodecahedron is a subgroup of a symmetric group $S_{20}$ on the 20 vertices, which is isomorphic to $A_5 \times C_2$. 
The vertices are indexed as Figure \ref{fig:Dodeca}. 
Let $N_{\sigma}$ denote the number of 
the 4-distance sets fixed by $\sigma \in {\rm Aut}(\Gamma)$. 
For each $\sigma \in {\rm Aut}(\Gamma)$, we determine the number $N_{\sigma}$. 

  \begin{figure}
   \begin{center}
        \includegraphics[width=60mm]{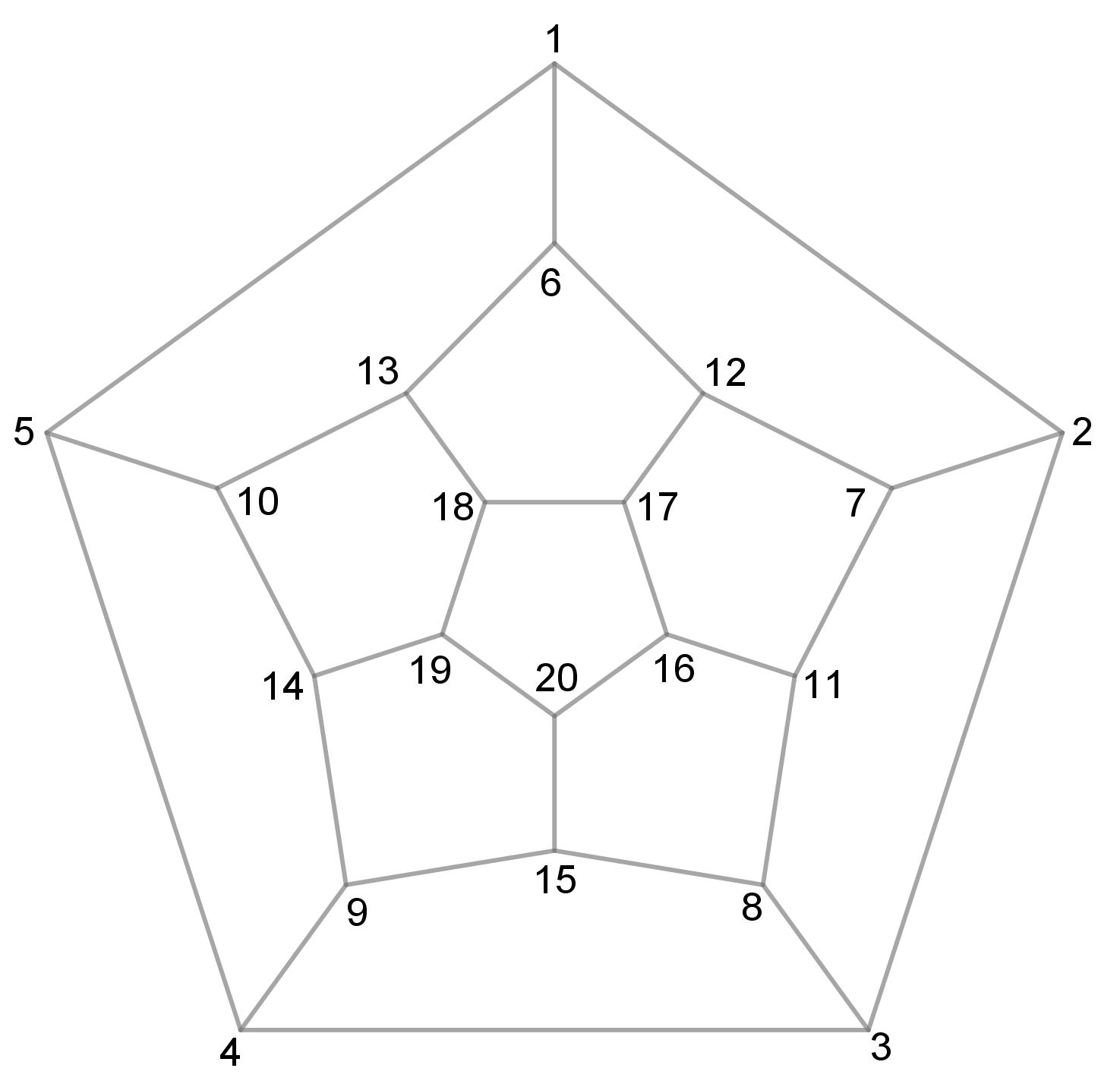}
    \end{center}
  \caption{Dodecahedron graph}
  \label{fig:Dodeca}
\end{figure}

The identity $e$ fixes all the 4-distance sets, namely $N_{e}=\binom{10}{2}\cdot 2^8=11520$. 

The transformations that fix a face of the dodecahedron are conjugates of 
\[\sigma_1=(1\,2\, 3\, 4\, 5)(6\,7\, 8\, 9\, 10)(11\,15\, 14\, 13\, 12)(16\,20\, 19\, 18\, 17),\]
$\sigma_1^2$, $\sigma_1^3$, or $\sigma_1^4$. 
The number of the transformations is 24. 
The size of a subset fixed by $\sigma_1$  is divisible by 5. Thus $N_{\sigma_1}=0$, and similarly $N_\sigma=0$ for any transformation $\sigma$ in this case. 

The transformations that fix an edge of the dodecahedron are conjugates of 
\[\sigma_2 =(1\,2)(3\,6)(5\,7)(4\,12)(10\,11)(8\,13)(9\,17)(14\,16)(15\,18)(19\,20).\]  
The number of the transformations is 15. 
A 4-distance set fixed by $\sigma_2$ contains 
one of $ \{1,2 \}$ and $\{19,20 \}$, 
one of $ \{3,6 \}$ and  $\{15,18 \}$,
one of $ \{5,7 \}$ and  $\{14,16 \}$, and 
one of $ \{4,12 \}$ and $\{9,17 \}$.  
 This implies that $N_{\sigma_2}=2^4=16$, and similarly $N_\sigma=16$ for any transformation $\sigma$ in this case.

The transformations that fix a vertex of the dodecahedron are conjugates of 
\[\sigma_3=(2\, 5\, 6)(3\, 10\, 12)(4\, 13\, 7)(8\, 14\, 17)(9\, 18\, 11)(15\, 19\, 16) \]
or $\sigma_3^2$.
The number of the transformations is 20. 
The size of a subset fixed by $\sigma_3$  is congruent to 0 or 1 modulo 3. Thus $N_{\sigma_3}=0$, and similarly $N_\sigma=0$ for any transformation $\sigma$ in this case. 

Let $\tau$ be the transformation such that $\tau(x)=-x$ for any vertex $x$, namely
\[
\tau=(1\,20)(2\,19)(3\,18)(4\,17)(5\,16)(6\,15)(7\,14)(8\,13)(9\,12)(10\,11). 
\]
Clearly $N_\tau=0$. 

We consider the transformations that are conjugates of  
\[\tau \sigma_1=(1\,19\,3\,17\,5\,20\,2\,18\,4\,16)(6\,14\,8\,12\,10\,15\,7\,13\,9\,11),\]
$\tau \sigma_1^2$, $\tau \sigma_1^3$, or $\tau \sigma_1^4$.
The number of the transformations is 24. 
The size of a subset fixed by $\tau\sigma_1$  is divisible by 10. Thus $N_{\tau \sigma_1}=0$, and similarly $N_\sigma=0$ for any transformation $\sigma$ in this case. 

We consider the transformations that are conjugates of  
\[\tau \sigma_2=(1\,19)(2\,20)(3\,15)(4\,9)(5\,14)(6\,18)(7\,16)(12\,17).\]  
The number of the transformations is 15. 
A 4-distance set fixed by $\tau \sigma_2$ may contain
one of $\{1,19\}$ and $\{2,20\}$, 
one of $ \{3,15 \}$ and $\{6,18 \}$,
one of $ \{5,14 \}$ and $\{7,16 \}$,  
one of $ \{4,9 \}$ and $\{12,17 \}$, 
one of $8$ and $13$, or one of $10$ and $11$.  
 This implies that $N_{\tau \sigma_2}=2^4+\binom{4}{3} 2^5=144$, and similarly $N_\sigma=144$ for any transformation $\sigma$ in this case. 
 
 We consider the transformations that are conjugates of  
\[\tau \sigma_3=(1\, 20)(2\,16\,6\,19\,5\,15)(3\,11\,12\,18\,10\,9)(4\,8\,7\,17\,13\,14)\]
or $\tau\sigma_3^2$. 
The number of the transformations are 20. 
A subset fixed by $\tau \sigma_3$ must contain $-x$ for its point $x$. 
Thus $N_{\tau \sigma_3}=0$, and similarly $N_\sigma=0$ for any transformation $\sigma$ in this case. 
 
By Burnside's lemma, the number of 8-point 4-distance subsets of the dodecahedron is 
\[
\frac{1}{|{\rm Aut}(\Gamma)|}\sum_{\sigma \in {\rm Aut}(\Gamma)} N_{\sigma}=
\frac{1}{120}(1\cdot 11520+24 \cdot 0+15\cdot 16+ 20 \cdot 0+ 
1\cdot 0+24 \cdot 0+15\cdot 144+ 20 \cdot 0)=116.  \qedhere
\]
\end{proof}

\section{Diameter graphs and their independence numbers} 

We denote a path and a cycle with $n$ vertices by $P_n$ and $C_n$, respectively. 
We denote a complete graph of order $n$ by $K_n$.
For $X\subset \mathbb{R}^d$, the {\it diameter} of $X$ is defined to be the maximum value of $A(X)$.  
Diameters give us important information when we study distance sets 
especially in few dimensional space. 
The {\it diameter graph} $DG(X)$ of $X\subset \mathbb{R}^d$ is the graph 
with $X$ as its vertices and where two vertices $p, q\in X$ are 
adjacent if $d(p,q)$ is the diameter of $X$. 
Let $R_n$ be the set of the vertices of a regular $n$-gon. 
Clearly $DG(R_{2n+1})=C_{2n+1}$ and $DG(R_{2n})=n\cdot P_2$. 
Note that if the independence number $\alpha (DG(X))=n'$ for an $s$-distance set $X$, 
then the subset of $X$ corresponding to an independence set of order $n'$ is an  $s'$-distance set for some $s'<s$. 

For diameter graphs for $\mathbb{R}^2$, we have the following propositions \cite{S08}. 

\begin{proposition}\label{prop:DGraphR2} 
Let $G=DG(X)$ for $X\subset \mathbb{R}^2$. Then 
\begin{itemize}
\item[{\rm (1)}] $G$ contains no $C_{2k}$ for any $k\geq 2$; 
\item[{\rm (2)}] if $G$ contains $C_{2k+1}$, 
then any two vertices in $V(G)\setminus V(C_{2k+1})$ are not adjacent and 
every vertex not in the cycle is adjacent to at most one vertex of the cycle. 
\end{itemize}
Moreover, $G$ contains at most one cycle. 
\end{proposition}

\begin{proposition}\label{independent} 
Let $G=DG(X)$ be the diameter graph of $X\subset \mathbb{R}^2$ with $|X|=n$. 
If $G\ne C_{n}$, then we have $\alpha(G)\geq \lceil \frac{n}{2} \rceil$. 
\end{proposition}

Propositions \ref{prop:DGraphR2} and \ref{independent}  are implied from the fact that two segments with the diameter 
must cross if they do not share an end point. 

For the diameter graphs of sets in $\mathbb{R}^3$, 
Dol'nikov \cite{Dol} proved the following theorem. 
This theorem plays a key role of the proof of Theorem~\ref{thm:sub}. 

\begin{theorem}[Dol'nikov] \label{graph}
Let $G=DG(X)$ be the diameter graph of $X\subset \mathbb{R}^3$. 
If $G$ contains two cycles with odd lengths, then they have a common vertex. 
\end{theorem}

In particular, we have the following corollary. 
\begin{corollary}\label{cycle}
Let $G=DG(X)$ be the diameter graph of $X\subset \mathbb{R}^3$ with $|X|=n$. 
If $G$ contains an odd cycle $C$ with length $m$, then $\alpha (G)\geq \lceil \frac{n-m}{2} \rceil$. 
\end{corollary}
\begin{proof}
If we remove the odd cycle $C$ from $G$, then any odd cycle in $G$ is 
broken by Theorem~\ref{graph}. This implies $G-C$ is a bipartite graph. Therefore 
we have $\alpha (G)\ge \alpha(G-C)\ge \lceil \frac{n-m}{2} \rceil$.
\end{proof}

In the remaining of this section, we give a proof of Theorem~\ref{thm:sub}. 

By Corollary~\ref{cycle}, if the diameter graph $G=DG(X)$ of $X\subset \mathbb{R}^3$ with $20$ points contains a $3$-cycle or a $5$-cycle, then $\alpha (G)\ge 8$. 
Let $\mathscr{G}_n$ be the set of all graphs of order $n$ 
which do not contain neither a $3$-cycle nor a $5$-cycle. We define 
\[f(n)=\min \{\alpha (G)\mid G\in \mathscr{G}_n\}.\]

Since $\alpha (C_n)=\lceil n/2\rceil$, we have $f(n)\le \lceil n/2\rceil$. 
For a group $G$ and $S\subset G$, we define the {\it Cayley graph} ${\rm Cay}(G,S)$ as the graph whose 
vertex set is $G$ and two vertices $v ,w \in G$ are adjacent 
if $v^{-1}w\in S$.  
It is easy to see that ${\rm Cay}(\mathbb{Z}_{17},\{\pm 1, \pm 6\})$ does not contain 
neither a $3$-cycle nor a $5$-cycle and 
$\alpha ({\rm Cay}(\mathbb{Z}_{17},\{\pm 1, \pm 6\}))=7$. 
This implies $f(17)\le 7$. 

\begin{lemma}\label{lem:fn}
Let $f(n)$ be defined as above. Then $0 \le f(n+1)-f(n) \le 1$ holds. 
\end{lemma}
\begin{proof}
Let $G \in \mathscr{G}_{n}$ be a graph satisfying $\alpha(G)=f(n)$ 
and $G'$ be the graph given by adding one isolated vertex to $G$. 
Then $f(n+1)\le \alpha(G')=\alpha(G)+1\le f(n)+1$. 
Let $G \in \mathscr{G}_{n+1}$ be a graph satisfying $\alpha(G)=f(n+1)$ 
and $H$ be an independent set of $G$ with $|H|=f(n+1)$. 
Let $v\in V(G)\setminus H$. Then $H$ is an independent set of $G-\{v\}$, and $\alpha (G-\{v\})=|H|$. 
Therefore $f(n)\le \alpha (G-\{v\})=|H|=f(n+1)$. 
\end{proof}

For a vertex $v\in V(G)$, $\Gamma _i(v)=\{w\in V(G) \mid \mathfrak{d}(v,w)=i\}$, where $\mathfrak{d}(v,w)$ is the shortest-path distance between $v$ and $w$. 
We abbreviate $\Gamma (v)=\Gamma _1(v)$. 
We define $G_i(v)$ as the induced subgraph with respect to $\Gamma _i(v)$ 
and $k_i(v)=|\Gamma _i(v)|$.  
Let $m$ be a positive integer. 
We denote $\Gamma _m^\ast (v)=\bigcup _{i\ge m}\Gamma_i (v)$ and $k_m^\ast (v)=|\Gamma _m^\ast (v)|$. 
Moreover, we define $G_m^\ast (v)$ as the induced subgraph of $G$ 
with respect to $\Gamma _m^\ast (v)$. 
Note that we regard $\mathfrak{d}(v,w)=\infty$ and $w\in \Gamma _m^\ast (v)$ if there is no path between $v$ and $w$. 
Then the following degree condition holds. 

\begin{lemma}\label{deg_condition}
Let $n$ and $t$ be positive integers. 
Let $G\in \mathscr{G}_n$ and $v\in V(G)$. 
If $\alpha (G)<t<n-k_1(v)+1$, then 
\[k_1(v)+f(n-k_1(v)-t+1)<t.\]
\end{lemma}
\begin{proof}
Since $G\in \mathscr{G}_n$, $\{v\}\cup \Gamma _2(v)$ is an independent set of $G$. 
In particular, we have $k_3^\ast (v)\ge n-k_1(v)-t+1$ since
$1+k_2(v)\le \alpha(G)<t $ and 
$k_2(v)\le t-2$. 
Then 
\[t>\alpha (G)\ge k_1(v)+f(k_3^\ast (v))\ge k_1(v)+f(n-k_1(v)-t+1),\]
since $w_1$ and $w_2$ are not adjacent for any $w_1\in \Gamma_1 (v)$ 
and $w_2\in \Gamma _3^\ast (v)$.  
\end{proof}

For a small integer $n$, we can determine $f(n)$ by using Lemma~\ref{deg_condition}
\begin{lemma}\label{table_f}
We have 
$f(3)=2, f(5)=3, f(8)=4, f(10)=5, f(13)=6$ and $f(17)=7$.
\end{lemma} 
\begin{proof}
 Since $P_3$, $P_5$, $C_8$, $C_{10}$, $C_{13}$ and 
${\rm Cay}(\mathbb{Z}_{17},\{\pm 1, \pm 6\})$ 
are examples whose independence numbers are the values $t$ in the assertion. 
This implies the inequalities $f(n)\le t$ for each case.
It is enough to prove the converse inequalities $f(n)\ge t$. 
We only prove $f(17)\ge 7$ because other inequalities can be proved by a similar way. 
Suppose that there exists $G \in \mathscr{G}_{17}$ such that $\alpha (G)<7$. 
If there exists $v\in V(G)$ such that $k_1(v)=3$, 
then $7> 3+f(8)=3+4$ from Lemma~\ref{deg_condition} with $t=7$, which is a contradiction. 
It is easy to see that we conclude a contradiction for $k_1(v)>3$ 
since $k_1(v)+f(n-k_1(v)-t+1)\leq (k_1(v)+1)+f(n-(k_1(v)+1)-t+1)$ 
holds in general from $f(n+1)-f(n)\le 1$ in Lemma~\ref{lem:fn}. 
Therefore $k_1(v)\le 2$ for any $v\in V(G)$. 
Then $G$ is the union of cycles, paths or isolated vertices. 
Except for an odd cycle, each connected component of $G$ with $m$ 
vertices has an independent set of size $\lceil m/2\rceil$. 
Moreover, $G$ contains at most one odd cycle by Theorem~\ref{graph}. 
Then $\alpha (G)\ge 8$, which is a contradiction. 
Therefore $f(17)\ge 7$ holds. Then we have $f(17)=7$. 
\end{proof}

We can determine other $f(n)$ for small $n$. 
For example we have $f(12)=5$ and $f(16)=7$. 
For $f(12)=5$, it is clear because $5=f(10)\le f(12)$ and 
${\rm  Cay}(\mathbb{Z}_{12},\{\pm 1, 6\})$ has the independence number 5. 
For $f(16)=7$, it is proved by a similar way to the proof of $f(17)=7$, and an attaining graph is obtained by removing one vertex from ${\rm  Cay}(\mathbb{Z}_{17},\{\pm 1, \pm 6\})$ while maintaining the independence number.  
However, the values in Lemma~\ref{table_f} are enough to prove Lemmas~\ref{disconnected} and \ref{diameter:20}.

\begin{lemma}\label{disconnected}
Let $G$ be the diameter graph of $X\subset \mathbb{R}^3$ with $|X|=20$. 
If $G$ is disconnected, then $\alpha (G)\ge 8$. 
\end{lemma}
\begin{proof}
Since $G$ is disconnected, there exists a partition $V=V_1\cup V_2$ such that 
$v_1$ and $v_2$ are not adjacent for any  $v_1\in V_1$ and $v_2\in V_2$. 
We may assume $|V_1|\ge 10$. Let $n_i=|V_i|$ and  
$H_i$ be the induced subgraph of $G$ with respect to $V_i$ for $i=1,2$. 
Note that we may assume that both $H_1$ and $H_2$ do not contain 
neither a $3$-cycle nor a $5$-cycle by Corollary~\ref{cycle}. 
If $10\le n_1\le 15$, then 
$\alpha (H_1)\ge f(n_1)\ge f(10)\ge 5$ and $\alpha (H_2)\ge f(n_2)\ge f(5) \ge 3$ 
by Lemma ~\ref{table_f}. 
Then $\alpha (G)= \alpha (H_1)+\alpha (H_2)\ge 5+3=8.$ 
If $n_1=16$, then $\alpha (G)\ge f(16)+f(4)\ge f(15)+f(3)= 6+2=8$ since $|H_1|\ge 16$ and $|H_2|\ge 4$. 
If $17\le n_1\le 19$, then $\alpha (G)\ge f(17)+f(1)\ge 7+1=8$ since $|H_1|\ge 17$. 
Therefore $\alpha (G)\ge 8$. 
\end{proof}

\begin{lemma}\label{diameter:20}
Let $G$ be the diameter graph of $X\subset \mathbb{R}^3$ with $|X|=20$. 
Then $\alpha (G)\ge 8$. 
\end{lemma}

\begin{proof}
Suppose that $G$ contains a $3$-cycle or a $5$-cycle. 
Then $\alpha (G)\ge 8$ holds by Corollary~\ref{cycle}. 
Therefore we may assume that $G$ does not contain neither a $3$-cycle nor a $5$-cycle. 
Let $v\in V(G)$. Since $G$ does not contain neither a $3$-cycle nor a $5$-cycle, 
both $\Gamma _1(v)$ and $\Gamma _2(v)\cup \{v\}$ are independent sets. 
Therefore we may assume $k_1(v)\le 7$ and $k_2(v)\le 6$. 
In particular, we may assume $k_3^\ast (v)=20-(1+k_1(v)+k_2(v))\ge 13-k_1(v)$. 
Moreover, we may assume that $G$ is connected by Lemma~\ref{disconnected}. 

Suppose $5\leq k_1(v) \leq 7$ for some $v\in V(G)$. Since $k_3^\ast (v)\ge 13-k_1(v) \ge 6$, we have $\alpha (G_3^\ast (v))\ge f(6) \ge f(5)=3$. 
Let $H$ be an independent set of $G_3^\ast (v)$ with $|H|=3$. 
Then $\Gamma _1(v)\cup H$ is an independent set of $G$. 
Therefore we have $\alpha (G)\ge k_1(v)+f(5)\ge  5+3=8$.
Suppose $k_1(v)=4$ for some $v\in V(G)$. 
Since $k_3^\ast (v) =13-k_1(v)\ge 9$, we have $\alpha (G)\ge k_1(v)+f(9)\ge k_1(v)+f(8) = 4+4=8$. 
Suppose $k_1(v)=3$ for some $v\in V(G)$. 
Then we have $k_3^\ast (v)=13-k_1(v)\ge 10$. 
Therefore $\alpha (G)\ge k_1(v)+f(10)=3+5=8$. 
Suppose $k_1(w)\le 2$ for any $w\in V(G)$. 
Since $k_1(w)\le 2$ and $G$ is connected, $G$ is isomorphic to $C_{20}$ or $P_{20}$.
Then we have $\alpha(G)=10$. 
This completes the proof. 
\end{proof}

By Lemma~\ref{diameter:20}, we have Theorem~\ref{thm:sub}.

\section{Colorings and their realizations} \label{sec:5}
Let $X=\{p_1, p_2, \ldots ,p_n\}$ be an $s$-distance set with $A(X)=\{\alpha_1, \alpha_2,\ldots, \alpha_s\}$. 
Let $[n]=\{1,2,\ldots ,n\}$ and $\binom{S}{k}=\{T\subset S\mid |T|=k\}$ for a finite set $S$. 
An $s$-distance set with $n$ points is represented by an edge coloring of the  complete graph $K_n$ by $s$ colors. 
We regard an $s$-coloring of the edge set of $K_n$ by a surjection $c:\binom{[n]}{2}\to [s]$.  
We define an {\it $s$-coloring} $c: \binom{[n]}{2}\to [s]$ of an $s$-distance set $X$ by a natural manner, namely, 
$c(\{i,j\})=k$ where $d(p_i,p_j)=\alpha_k$. 
Conversely, an $s$-distance set $X$ is called a {\it realization} of $c$ if 
$c$ is a coloring of $X$.

Two $s$-colorings $c_1$ and $c_2$ are said to be {\it equivalent} if there exists bijections 
$g:[s]\to [s]$ and $h:[n]\to [n]$ such that $g(c_1(\{h(i),h(j)\}))=c_2(\{i,j\})$ for each $\{i,j\}\in \binom{[n]}{2}$. 
We define the {\it coloring matrix} $C=C(x_1, x_2,\ldots, x_s)$ of the coloring $c$ with respect to  $x_1, x_2,\ldots, x_s$ by 
\[C_{i,j}=\begin{cases}0 & {\rm if }\  i=j, \\
x_{c(\{i,j\})} & {\rm if } \  i\ne j.\end{cases}\] 
In particular, $C=C(1,2,\ldots ,s)$ is called a {\it normal coloring matrix}. 
A coloring $c$ is often represented as its normal coloring matrix $C$ in this paper.
We distinguish them by  lowercase letter $c$ and uppercase letter $C$.

For a subset $X=\{p_1, p_2, \ldots ,p_n\}\subset \mathbb{R}^d$, we define the {\it squared distance matrix} $D=D(X)$ of $X$ by 
\[D=(d(p_i, p_j)^2)_{1\le i,j\le n}.\] 
For an $n\times n$ symmetric matrix $M=(m_{i,j})_{1\le i,j\le n}$, 
${\rm Gram}(M)$ is defined to be the $(n-1)\times (n-1)$ symmetric matrix with $(i,j)$ entries 
\[ \frac{m_{i,n}+m_{j,n}-m_{i,j}}{2}. 
\] 
For $X=\{p_1, p_2, \ldots ,p_n\}$, the matrix ${\rm Gram}(D(X))$ is the Gram matrix of $X$ 
when $p_n$ is located at the origin.  

\begin{theorem}
Let $M$ be an $(n-1)\times (n-1)$ real symmetric matrix.  
There exists $X$ in $\mathbb{R}^d$ such that 
$M$ is equal to the Gram matrix of $X$ 
if and only if $M=(m_{i,j})_{1\le i,j\le n-1}$ satisfies 
\begin{equation}\label{eq:realize}
\begin{cases}
M {\rm \ is\ positive\ semidefinite}, \\
\rank M\le d, \\
m_{i,i}>0 {{\rm \ for\ every\ }} i\in [n-1]  {\rm \ and \ } \\
m_{i,j}<\frac{m_{i,i}+m_{j,j}}{2} {{\rm \ for\ every\ }}  1\le i<j\le n-1. 
\end{cases}
\end{equation}
\end{theorem}

An $s$-coloring $c:\binom{[n]}{2}\to [s]$ is said to be {\it representable} in $\mathbb{R}^d$ if 
there exists distinct real numbers $\alpha _1, \alpha _2, \ldots , \alpha _s$ 
such that $C(\alpha_1,\alpha _2, \ldots , \alpha _s)$ satisfies (\ref{eq:realize}). 
To decide an $s$-coloring are not representable, the rank condition in (\ref{eq:realize}) is effective. 
An $s$-coloring $c:\binom{[n]}{2}\to [s]$ is said to be {\it quasi representable} in $\mathbb{R}^d$ if 
there exists distinct complex numbers $\alpha _1, \alpha _2, \ldots , \alpha _s$ 
such that $\rank C(\alpha_1,\alpha _2, \ldots , \alpha _s)\le d$. 

For a square matrix $M=(m_{i,j})_{1\le i,j\le n}$ and an index set $T=\{t_1, t_2, \ldots ,t_k\}\in \binom{[n]}{k}$, 
we define a {\it principal submatrix} of $M$ with respect to $T$ by 
\[ {\rm sub} (M;T)=(m_{t_i,t_j})_{1\le i,j\le k}.\]
Let 
\[\mathcal{M}_k(M)= \left\{{\rm sub} (M;T)\mid T\in \binom{[n]}{k}\right\}.\]

We define 
\[r(M)=\max \left\{ k\in [n]\mid  \exists S \in \mathcal{M}_{k}, \det S \ne 0 \right\}. \]

\begin{proposition}
For a square matrix $M$, $r(M) \le \rank M$ holds. Moreover, $r(M)=\rank M$  holds if 
$M$ is positive semidefinite. 
\end{proposition}
\begin{proof}
It is well known that $\rank M$ is the maximum value $k$ such that  
there exists a square submatrix $S$ of size $k$ in $M$ with $\det S \ne 0$ that may not be principal. 
This implies $r(M) \le \rank M$. Suppose $M$ is a positive semidefinite matrix of size $n$. 
 Since $M$ is positive semidefinite, there exists $n \times \rank M$ matrix $N$ such that $M=NN^\top$ and $\rank N=\rank M$.  
For $T \in \binom{[n]}{k}$, let $\chi_{T}$ be 
the $n\times n$ diagonal matrix with diagonal entries $(\chi_{T})_{ii}=1$ if $i \in T$, and $(\chi_{T})_{ii}=0$ if $i \not\in T$.  
For a row vector $\mathbf{x} \in \mathbb{R}^n$ and $T \in \binom{[n]}{k}$, it follows that  $\mathbf{x} (\chi_{T} N N^\top \chi_{T}^\top)\mathbf{x}^\top=0$ if and only if
$\mathbf{x}(\chi_{T} N)=0$. Thus, 
\begin{align*}
\rank(\chi_{T} N N^\top \chi_{T}^\top)&=n-\dim \{\mathbf{x}\in \mathbb{R}^n \mid 
\mathbf{x} (\chi_{T} N N^\top \chi_{T}^\top)\mathbf{x}^\top=0\}\\
&=n-\dim \{\mathbf{x} \in \mathbb{R}^n \mid \mathbf{x}(\chi_{T} N)=0\}=\rank (\chi_T N).
\end{align*}
For $T \in \binom{[n]}{k}$, it follows that $\det ({\rm sub}(M;T))\ne 0$ if and only if $\rank (\chi_{T} N N^\top \chi_{T}^\top)\geq k$. 
This implies that $r(M)$ is the maximum value $k$ such that $\rank(\chi_{T} N)=k$, which is $k=\rank N=\rank M$.  
\end{proof}

An $s$-coloring $c:\binom{[n]}{2}\to [s]$ is said to be {\it weakly quasi representable} in $\mathbb{R}^d$ if 
there exist distinct complex numbers $\alpha _1, \alpha _2, \ldots , \alpha _s$ such that 
$r(C(\alpha _1,\alpha_2,\ldots ,\alpha_s))\le d$. 
Clearly if an $s$-coloring $c$ is representable in $\mathbb{R}^3$, then 
$c$ is a (weakly) quasi representable in $\mathbb{R}^3$. 
The following proposition is essentially proved by Sz\"{o}ll\H{o}si and \"{O}sterg{\aa}rd \cite{SO20}, but 
we should take all submatrices $M$ that may not be principal in their result. 
Actually, it is enough to use all principal submatrices $M$ to collect our desired colorings. 
\begin{proposition}\label{prop:equivalent}
An $s$-coloring $c:\binom{[n]}{2}\to [s]$ is a weakly quasi representable in $\mathbb{R}^3$ if and only if 
the following system of equations in $s$ values 
\begin{equation}\label{eq:wqrc}
    \begin{cases}
        \det M=0 \quad {{\rm for \ all}}\ M\in \mathcal{M}_4({\rm Gram}(C(1,{x_1,x_2,\ldots ,x_{s-1}}))),\\
        \quad 1+u\prod_{i=1}^{s-1} x_i(x_i-1)\prod_{1\le j<k\le s-1} (x_j-x_k)=0 
    \end{cases}
\end{equation}
has a complex solution.
\end{proposition}

\section{$5$-distance sets containing $8$-point $s$-distance sets for $s\le 4$}

By Theorem~\ref{thm:sub}, to classify $20$-point $5$-distance sets in $\mathbb{R}^3$, 
it is enough to consider $5$-distance sets which contain $8$-point $s$-distance sets in $\mathbb{R}^3$ for $s\le 4$. We will prove the following theorem in this section. 

\begin{theorem}\label{thm:contain8}
Let $Y$ be an $s$-distance set in $\mathbb{R}^3$ with $8$ points for $3\le s\le 4$. 
If $Y\cup Z$ is a $5$-distance set in $\mathbb{R}^3$ with at least $20$ points, 
then $Y\cup Z$ is isomorphic to a regular dodecahedron. 
\end{theorem}
Note that there exists no 2-distance set with 8 points in $\mathbb{R}^3$. 
In this section, firstly, we consider (weak) quasi representable $s$-colorings $c$ 
in $\mathbb{R}^3$ instead of $s$-distance sets in $\mathbb{R}^3$. 
Then we consider realizations of $c$ as needed. 

Sz\"{o}ll\H{o}si and \"{O}sterg{\aa}rd \cite{SO20} classified 
quasi representable $s$-colorings in $\mathbb{R}^3$ for $s\le 4$, see Table~\ref{tab:SO2020_2}. 

\begin{lemma}[Sz\"{o}ll\H{o}si and \"{O}sterg{\aa}rd \cite{SO20}]\label{CG}
There exist exactly $19$ quasi representable $3$-colorings in $\mathbb{R}^3$ 
with $8$ vertices and  exactly
$1074$ quasi representable $4$-colorings in $\mathbb{R}^3$ with $8$ vertices. 
\end{lemma}

\begin{table} 
\begin{center}
\begin{tabular}{ccccccccccc}
\hline 
$n$&{$5$}&{$6$}&$7$&$8$&$9$&$10$&$11$&$12$&$13$&$14$\\
\hline 
$\sharp$ QRC & $512$ & {$62095$} & {$4499$} & {$1093$} & $277$ & $59$ & $12$ & $5$ & $2$ & $0$  \\
\hline
\end{tabular}
\caption{Number of quasi representable at most $4$-colorings in $\mathbb{R}^3$ of $n$ points}
\label{tab:SO2020_2}
\end{center}
\end{table}

We denote the set of all quasi representable $s$-colorings 
in $\mathbb{R}^3$ with $n$ vertices by $\mathcal{CG}(n,s)$. 
By Lemma~\ref{CG}, we have $|\mathcal{CG}(8,3)|=19$ and $|\mathcal{ CG}(8,4)|=1074$. 

Let $C\in \mathcal{CG}(8,3)\cup \mathcal{CG}(8,4)$. 
We define a graph $G(C)=(V,E)$ with respect to $C$ as follows. 
By a computer search, we find all vectors $(a_1, a_2,\ldots ,a_8)\in [5]^8$ 
such that 
\begin{equation}\label{nine}
M=\left(\begin{array}{ccccc}
0&a_1&\cdots &a_8\\
a_1&  &&&\\
\vdots  &&C&\\ 
a_8&  &&\\
\end{array}
\right)
\end{equation}
is a weakly quasi representable $s$-colorings in $\mathbb{R}^3$ for $s\le 5$ by Proposition~\ref{prop:equivalent}. In order to check whether $M$ is weakly quasi representable, first we calculate a Gr\"{o}bner basis $\mathfrak{B}$ of system \eqref{eq:wqrc} for $C$, see \cite{SO20} about the manner in details. 
Then, we calculate a Gr\"{o}bner basis for 
the union of $\mathfrak{B}$ and the set of the first equations in \eqref{eq:wqrc} for all ${\rm sub}(M;T)$ with $1 \in T$,
which determine whether $M$ is weakly quasi representable. 
Throughout this paper, computer calculations are done with functions of the software Magma \cite{magma} and Maple \cite{maple}.   

We regard the set of all vectors satisfying (\ref{nine}) as the vertex set $V$ of the graph $G(C)$. 
Then two vertices $(a_1, a_2,\ldots ,a_8), (a'_1, a'_2,\ldots ,a'_8)\in V$ are adjacent if 
there exists $i\in [5]$ such that 
\begin{equation}\label{ten}
\left(\begin{array}{cccccc}
0 & i &a_1&\cdots &a_8\\
i & 0 &a'_1&\cdots &a'_8\\
a_1& a'_1  &&&\\
\vdots  & \vdots&&C&\\ 
a_8&  a'_8 &&&\\
\end{array}
\right)
\end{equation}
is a weakly quasi representable $s$-coloring in $\mathbb{R}^3$ for $s\le 5$. 
Some special graphs have loops as Lemma~\ref{lem:loop} below.   
For positive real numbers $\alpha_1, \alpha_2,\ldots ,\alpha_n$ and a subset $X=\{p_1, p_2, \ldots , p_n\}\subset \mathbb{R}^3$ which 
is not co-linear, 
there exist at most two points 
$q\in \mathbb{R}^3$ such that $d(p_i,q)=\alpha _i$ for any $i\in [n]$. 
In particular, if there exist two points which satisfy this condition, 
then $X$ is co-planar.  
By exhaustive computer search, we have the following lemma. 

\begin{lemma}\label{lem:loop}
Let $C\in \mathcal{CG}(8,3)\cup \mathcal{CG}(8,4)$. 
$G(C)$ has a loop if and only if a realization of $C$ is 
isomorphic to one of the following nine $4$-distance sets.
\begin{itemize}
\item[{\rm (a)}] the subset with $8$ points of a regular nonagon, 
\item[{\rm (b)}] a regular octagon, 
\item[{\rm (c)}] the six subsets with $8$ points of the set in Figure 1 (a) in Theorem~\ref{thm:known}, 
\item[{\rm (d)}] the set in Figure 1 (d) in Theorem~\ref{thm:known}.  
\end{itemize} 
Moreover, there exist two loops only for (d) and there is only one loop for other cases. 
\end{lemma}

For $C\in \mathcal{CG}(8,3)\cup \mathcal{CG}(8,4)$,  
let $\omega^\ast (C)=\omega(G(C))+l(G(C))$,  
where $\omega(G(C))$ is the clique number of $G(C)$ and $l(G(C))$ is the number of loops in $G(C)$. 
To prove Theorem~\ref{thm:contain8}, it is enough to classify 
$C\in \mathcal{CG}(8,3)\cup \mathcal{CG}(8,4)$ such that 
$\omega^\ast(C)\ge 12$. 
By exhaustive computer search, we have the following lemma. 

\begin{lemma}\label{lem:QRC} 
\begin{itemize}
    \item[{\rm (i)}] There exists a unique coloring $C\in \mathcal{CG}(8,3)$ with $\omega^\ast (C)\ge 12$, which corresponds to the cube. Moreover, $\omega ^\ast  (C)=\omega(G(C))= 12$ and there exists the unique clique of order $12$ for the coloring. 
    \item[{\rm (ii)}] There exist exactly $63$ colorings  $C\in \mathcal{CG}(8,4)$ with $\omega^\ast  (C)\ge 12$. Moreover, $\omega ^\ast  (C)=\omega(G(C))= 12$ and there exists the unique clique of order $12$ for each coloring among the $63$ colorings. 
\end{itemize}
\end{lemma}

We classify $8$-point $s$-distance sets for $s\le 4$ 
which are realizations of quasi representable $s$-colorings in Lemma~\ref{lem:QRC}. 
If a realization $X$ of a coloring $C$ is a subset of the dodecahedron, then we have another realization $\hat{\Phi}(X)$ of $C$ by Lemma~\ref{lem:structures}. The realizations $X$ and $\hat{\Phi}(X)$ may be isomorphic. 

Let $C\in \mathcal{CG}(8,3)$ be the coloring in Lemma~\ref{lem:QRC} (i) 
and $X$ be a realization of $C$. 
Then we have $A(X)=\{1, \sqrt{2}, \sqrt{3}\}$ by solving system (\ref{eq:wqrc}). 
Then it is easy to see that $X$ is a cube. 

Let 
\[
  C_1 = \left(
    \begin{array}{cccccccc}
      0 & 1 & 2 & 2 & 1 & 3 & 3 & 1 \\
      1 & 0 & 1 & 2 & 2 & 2 & 3 & 2 \\
      2 & 1 & 0 & 1 & 2 & 1 & 2 & 3 \\
      2 & 2 & 1 & 0 & 1 & 2 & 1 & 3 \\
      1 & 2 & 2 & 1 & 0 & 3 & 2 & 2 \\
      3 & 2 & 1 & 2 & 3 & 0 & 2 & 4 \\
      3 & 3 & 2 & 1 & 2 & 2 & 0 & 4 \\
      1 & 2 & 3 & 3 & 2 & 4 & 4 & 0 \\
    \end{array}
  \right). 
\]
Then $C_1\in \mathcal{CG}(8,4)$ is a coloring in Lemma~\ref{lem:QRC} (ii). 
There exist four solutions of system (\ref{eq:wqrc}) for $C_1$ and 
there exist exactly four realizations of $C_1$ up to isomorphism. 
Let 
\begin{equation} \label{eq:w2}
    W=\{6,12,17,18,13,16,19,1\}
\end{equation}
be a subset of the vertex set of the dodecahedron graph as given in Figure~\ref{fig:Dodeca}. The shortest-path distance matrix $(\mathfrak{d}(x,y))_{x,y \in W}$ of $W$ is $C_1$, where
$C_1$ is indexed by $W$ using the order of elements in \eqref{eq:w2}. 
Four realizations of $C_1$ are given in Figure~\ref{fig:sub8}. 
Let $Y_1=\{A_1,A_2,\ldots , A_7\}, Y_2=\{B_1,B_2,\ldots , B_7\}$ and 
\[
X_i=Y_i \cup \{P_1\}, \quad X_i'=Y_i \cup \{P_1'\} \qquad (i=1,2), 
\]
where $P_i'$ is the reflection of $P_i$ in the plane $\pi_i$ for $i=1,2$. 
Then $X_1,X_2,X_1'$ and $X_2'$ are all the realizations of $C_1$. 
Both $X_1$ and $X_2$ are subsets of the dodecahedron and have the structure of the coloring $C_1$, that shows the situation of Lemma~\ref{lem:structures}.  
There exist exactly two solutions of system (\ref{eq:wqrc}) for the other $62$ colorings $C\in \mathcal{CG}(8,4)$. 

If $C\in \mathcal{CG}(8,4)$ is one of the ten colorings obtained from

\begin{center}
\begin{tabular}{lll}\label{tab:subset}
$\{2,7,9,10,13,16,17,18 \},
$ & $\{1,2,9,11,14,16,17,18  \},
$ &  $\{1,8,12,15,16,17,18,19\}, 
$ \\ $\{3,4,6,12,13,16,19,20\}, 
$ & $\{3,4,6,11,12,13,16,20\},
$  & $\{6,8,10,12,14,16,19,20\},
$ \\ $\{3,11,13,14,15,16,19,20\}, 
$ & $\{1,12,13,15,16,17,18,19\},
$ & $\{2,9,11,13,14,15,16,20\}, 
$ \\$\{1,12,14,15,16,17,18,19\},
$ & & \\
\end{tabular}
\end{center}
\noindent
by the above manner, then the two realizations of $C$ are isomorphic. 
Except for the above 10 colorings and $C_1$, each $C\in \mathcal{CG}(8,4)$ in Lemma~\ref{lem:QRC} (ii) 
has exactly two realizations up to isomorphism. 
Then we have the following lemma. 

\begin{figure}
\begin{center}
        \includegraphics[width=58mm]{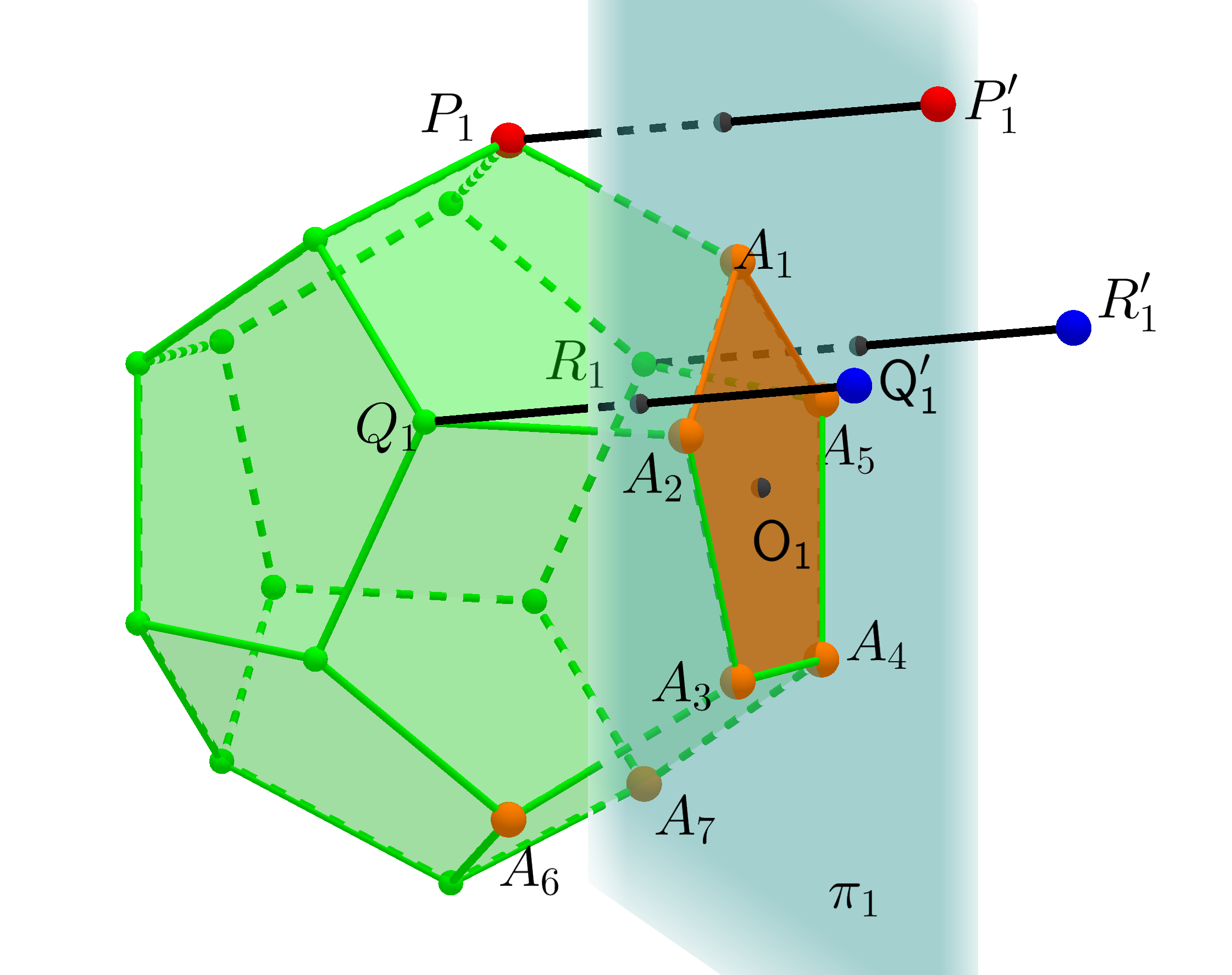}\qquad 
        \includegraphics[width=50mm]{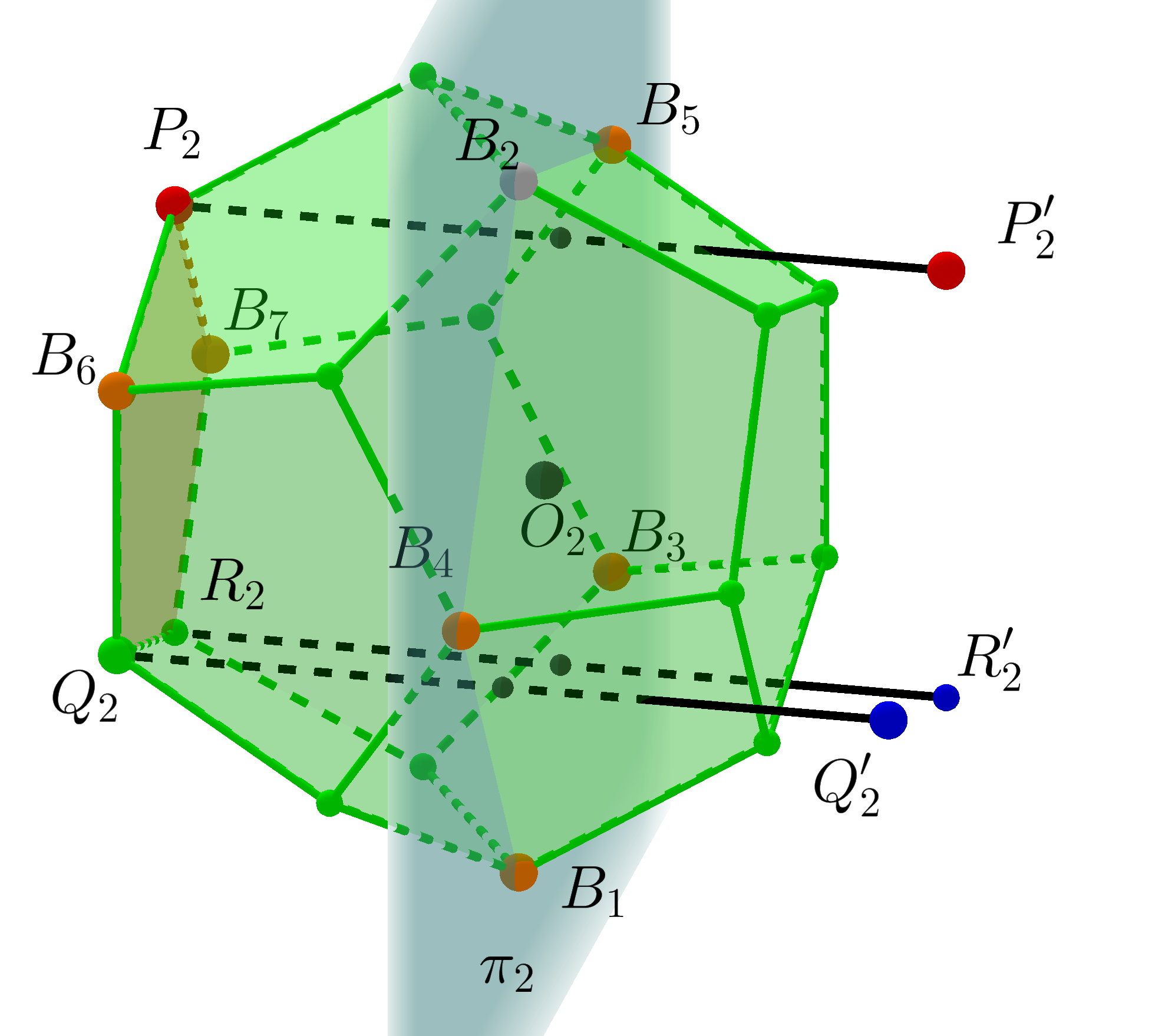}
\end{center}
   \caption{the $8$-point subsets which are realizations of $C_1$}
  \label{fig:sub8}
\end{figure}

\begin{lemma}\label{lem:realizations}
\begin{itemize}
    \item[{\rm (i)}] There exists a unique $3$-distance set whose 
    coloring is given in Lemma~\ref{lem:QRC}~(i). 
    \item[{\rm (ii)}] Among the colorings in Lemma~\ref{lem:QRC} (ii), we have the following:
    \begin{itemize}
        \item[{\rm (a)}] There exists exactly one coloring which has four solutions of system (\ref{eq:wqrc}) and the four realizations corresponding to the solutions are not isomorphic to each other.  
        \item[{\rm (b)}] There exist exactly $52$ colorings which have two solutions of system (\ref{eq:wqrc}) and the two realizations corresponding to the solutions are not isomorphic. 
        \item[{\rm (c)}] There exist exactly $10$ colorings which have two solutions of system (\ref{eq:wqrc}) but the two realizations corresponding to the solutions are isomorphic. 
    \end{itemize}
\end{itemize}
\end{lemma}

By Lemma~\ref{lem:realizations}, there exist exactly 118 of 4-distance sets in $\mathbb{R}^3$ given from the colorings in Lemma~\ref{lem:QRC}.  
The two sets $X_1'$ and $X_2'$ are not subsets of the dodecahedron. 
By Lemma~\ref{lem:SubDodeca4ds}, the remaining 116 4-distance sets should be  subsets of the dodecahedron. 
Note that the cube is also a subset of the dodecahedron. 
Let $S$ be the set of the cube and the 116 4-distance subsets. 
By Lemma~\ref{lem:QRC}, $\omega  (G(C))=12$ for the coloring $C$ obtained from $X \in S$, and the corresponding clique of order $12$ is unique.  
Therefore a 20-point 5-distance set that contains $X \in  S$ must be the dodecahedron. 

In order to prove Theorem~\ref{thm:contain8}, 
we prove that for $i=1,2$ there is no subset $Z \subset \mathbb{R}^3$ such that $X_i'\cup Z$ is a 20-point 5-distance set.
We consider a candidate $P\in \mathbb{R}^3$ such that $|A(X_i'\cup \{P\})|\le 5$ 
for $i=1,2$ by using (\ref{nine}). 
The candidates are $\{Q_1', R_1',O_1\}$ for $X_1'$ and $\{Q_2', R_2',O_2\}$ for $X_2'$ in Figure~\ref{fig:sub8}. The points $Q_i'$ and $R_i'$ are the reflections 
of $Q_i$ and $R_i$ in the plane $\pi_i$, respectively. 
Moreover, $O_1$ ({\it resp.} $O_1'$) is the center of the pentagon consisted of 
$\{A_1, A_2, \ldots, A_5 \}$ ({\it resp.} $\{B_1,B_2,\ldots , B_5\}$). 
Thus the cardinality of a 5-distance set that contains $X_i'$ is at most $11$. 
Therefore a proof of Theorem~\ref{thm:contain8} is complete.

Finally, we prove Theorem~\ref{thm:main}. 
\begin{proof}[Proof of Theorem~\ref{thm:main}]
By Theorem~\ref{thm:sub} and $g_3(2)=6$, every 5-distance set in $\mathbb{R}^3$ at least 20 points contains an 8-point $s$-distance set for some $3\leq s \leq 4$. 
 Therefore the assertion follows by Theorem~\ref{thm:contain8}.  
\end{proof}

\bigskip

\noindent
\textbf{Acknowledgments.} 
The authors thank Kenta Ozeki for providing information on the papers \cite{ES18,KO13} relating to graphs without two vertex-disjoint odd cycles. 
Nozaki is supported by JSPS KAKENHI Grant Numbers 
19K03445 and 20K03527. 
Shinohara is supported by JSPS KAKENHI Grant Number  18K03396.

\quad 

\noindent
{\it Hiroshi Nozaki}\\ 
	Department of Mathematics Education, 
	Aichi University of Education, 
	1 Hirosawa, Igaya-cho, 
	Kariya, Aichi 448-8542, 
	Japan.\\
E-mail address: {\tt hnozaki@auecc.aichi-edu.ac.jp}

\quad \\
\noindent
{\it Masashi Shinohara}\\
	Faculty of Education, 
	Shiga University,
	2-5-1 Hiratsu, Otsu, Shiga 520-0862, 
	Japan. \\
	E-mail address: {\tt shino@edu.shiga-u.ac.jp}
	

\begin{thebibliography}{99}


\bibitem{BBS}
E. Bannai, E. Bannai, and D. Stanton, 
An upper bound for the cardinality of an $s$-distance subset in real Euclidean space, II, 
\textit{Combinatorica} {\bf 3} (1983), 147--152.







\bibitem{B84}
A. Blokhuis, 
Few-distance sets, 
CWI Tract, {\bf 7} (1984), 1--70.

\bibitem{magma}
W. Bosma, J. Cannon, and C. Playoust, 
The Magma algebra system. I. The user language, {\it J. Symbolic Comput.} {\bf 24} (1997), 235--265.

\bibitem{BCNb}
A.E. Brouwer, A.M. Cohen, and A. Neumaier, 
{\it Distance-regular Graphs},  
Springer-Verlag, Berlin, (1989). 

\bibitem{CFG94}
H. T. Croft, K. J. Falconer, and R. K. Guy, 
{\it Unsolved Problems in Geometry, Problem Books in Mathematics},  Springer, New York (1994). 


\bibitem{DKT16}
E.R. van Dam, J.H. Koolen, and H. Tanaka, 
Distance-regular graphs, 
{\it Electron.\ J. Comb.} (2016), 
\#DS22. 

\bibitem{Dol} V. L. Dol'nikov,  
Some properties of graphs of diameters. The Branko Grünbaum birthday issue.  \textit{Discrete Comput. Geom.} \textbf{24}  (2000), 293–299. 


\bibitem{ES66}	
S. J. Einhorn and I. J. Schoenberg,
On Euclidean sets having only two distances between points. I. II, 
{\it Nederl.\ Akad.\ Wetensch.\ Proc.\ Ser.\ A} {\bf 69}={\it Indag.\ Math.} {\bf 28} (1966), 
479--488, 489--504.

\bibitem{EF96}
 P. Erd\H{o}s and P. Fishburn, Maximum planar sets that determine $k$ distances, 
{\it Discrete Math.} {\bf 160} (1996), 115--125.

\bibitem{ES18}
L. Esperet and M. Stehl\'{i}k, 
The width of quadrangulations of the projective plane, 
{\it J.\ Graph Theory} {\bf 89} (2018), 76--88. 


\bibitem{KO13}
K. Kawarabayashi and K. Ozeki, 
A simpler proof for the two disjoint odd cycles theorem, 
\textit{J. Combin.\ Theory, Ser.\ B} {\bf 103} (2013), 313--319.


\bibitem{LRS77}
D.G. Larman, C.A. Rogers, and J.J. Seidel, 
On $2$-distance sets in Euclidean space, 
\textit{Bull.\ London Math.\ Soc.} 
{\bf 9} (1977), 261--267.

\bibitem{L97}
P. Lison\v{e}k, 
New maximal $2$-distance sets, 
\textit{J. Combin.\ Theory, Ser.\ A} {\bf 77} (1997), 318--338.

\bibitem{maple}
Maple (2019). Maplesoft, a division of Waterloo Maple Inc., Waterloo, Ontario.

\bibitem{MS19} A. Munemasa and M. Shinohara, 
Complementary Ramsey numbers and Ramsey graphs, 
{\it J. Indones. Math. Soc.} {\bf 25} (2019), no. 2, 146–153. 





\bibitem{S04}
 M. Shinohara, 
Classification of three-distance sets in two dimensional Euclidean space, 
{\it European J.\ Combin.} {\bf 25} (2004),
1039--1058.

\bibitem{S08}
 M. Shinohara, 
Uniqueness of maximum planar five-distance sets, 
{\it Discrete Math.} {\bf 308} (2008), 3048--3055.

\bibitem{Spre}
M. Shinohara, 
Uniqueness of maximum three-distance sets in the three-dimensional Euclidean space, 
arXiv:1309.2047. 

\bibitem{SO20}
F. Sz\"{o}ll\H{o}si and P.R.J. \"{O}sterg{\aa}rd,
Constructions of maximum few-distance sets in Euclidean spaces, 
{\it Electron.\ J. Combin.} {\bf 27} (1)
(2020), \#P1.23.

\bibitem{W12}
X. Wei,
A proof of Erd\H{o}s--Fishburn's 
conjecture for $g(6)=13$, 
{\it Electron.\ J. Combin.} {\bf 19} (4)
(2012), \#P38.
\end{thebibliography}
\end{document}